\documentclass[11pt]{article}
\usepackage[english]{babel}
\usepackage{newlfont}
\usepackage{amsbsy}
\usepackage[dvips]{graphicx}
\usepackage{color}
\usepackage{bbm}
\usepackage{graphicx, subfigure}
\usepackage{color}
\usepackage[center]{caption2}
\usepackage{amssymb}
\usepackage{amsmath}
\usepackage{latexsym}
\usepackage{amsthm}
\usepackage{amsthm}
\theoremstyle{plain}
\newtheorem{thm}{Theorem}

\newtheorem{nota}[thm]{Notation}
\newtheorem{rem}[thm]{Remark}
\newtheorem{defin}[thm]{Definition}

\newcommand{\R}{\mathbb{R}}

\newcommand{\N}{\mathbb{N}}

\usepackage{mathtools}
\def\multiset#1#2{\ensuremath{\left(\kern-.2em\left(\genfrac{}{}{0pt}{}{#1}{#2}\right)\kern-.2em\right)}}
\usepackage[center]{caption2}

\begin{document}

\title{On graphlike $k$-dissimilarity vectors}
\author{Agnese Baldisserri \  \ \  \ \ Elena Rubei}
\date{}
\maketitle

\def\thefootnote{}
\footnotetext{ \hspace*{-0.36cm}

 }

\begin{abstract}
Let  ${\cal G}=(G,w)$ be a positive-weighted simple finite graph,
that is, let $G$ be a simple finite graph endowed with a function $w$
from the set of the edges of $G$ to the set of the positive real numbers.
For any subgraph $G'$ of  $G$, we define  $w(G')$ to be the sum of the weights 
of the edges of $G'$. 
For any  $i_1,..., i_k $ vertices of $G$, 
let $D_{\{i_1,.... i_k\}} ({\cal G})$ be
the minimum of the weights of the subgraphs of $G$ connecting $i_1,..., i_k$.
The  $D_{\{i_1,.... i_k\}} ({\cal G})$ are called 
$k$-weights of  ${\cal G}$. 

Given a family of positive real numbers parametrized by the $k$-subsets of $
\{1,..., n\}$, 
$\{D_I\}_{I  \in {\{1,...,n\} \choose k}}$, we can wonder when 
there exist a weighted graph ${\cal G}$ (or a weighted tree) 
and an  $n$-subset $\{1,..., n\}$ of  the set of its vertices such that 
$D_I ({\cal G}) =D_I $ for any $I  \in {\{1,...,n\} \choose k}$.
In this paper we study  this problem in the case $k=n-1$.

\end{abstract}

\section{Introduction}

For any graph $G$, let $E(G)$, $V(G)$ and $L(G)$ 
 be respectively the set of the edges,   
the set of the vertices and  the set of the leaves of $G$.
A {\bf weighted graph} ${\cal G}=(G,w)$ is a graph $G$ 
endowed with a function $w: E(G) \rightarrow \R$. 
For any edge $e$, the real number $w(e)$ is called the weight of the edge. If 
all the weights are nonnegative (resp. positive), 
we say that the graph is {\bf 
nonnegative-weighted} (resp. {\bf positive-weighted}),
 if all the weights are nonnegative and the ones of 
the internal edges are positive,
 we say that the graph is {\bf internal-positive-weighted}.
Throughout the paper we will consider only simple finite graphs.

For any subgraph $G'$ of  $G$, we define  $w(G')$ to be the sum of the weights 
of the edges of $G'$. 

\begin{defin}
Let ${\cal G}=(G,w) $ be a nonnegative-weighted graph. 
For any distinct $ i_1, .....,i_k \in V(G)$,
 we define $$ D_{\{i_1,...., i_k\}}({\cal G}) 
= min 
\{w(R) | \; R \text{ a connected subgraph of } G  \text{ such that } V(R) \ni 
i_1,...., i_k\}.$$ More simply, we denote 
$D_{\{i_1,...., i_k\}}({\cal G})$ by
$D_{i_1,...., i_k}({\cal G})$ for any order of $i_1,..., i_k$.  
We call  the  $ D_{i_1,...., i_k}({\cal G})$ 
the {\bf $k$-weights} of ${\cal G}$ and
 we call  a $k$-weight of ${\cal G}$ for some $k$  a {\bf multiweight}
 of ${\cal G}$.
\end{defin}

%Observe that in the case ${\cal G}$ is a tree,  $ D_{i_1,...., i_k}({\cal G})$ is the sum of the weights of the edges of the  minimal subtree joining $i_1,...,i_k$.

 If $S $ is a subset of $V(G)$ and  we order in some way the $k$-subsets of 
$ S$ (for instance, we order $S$ in some way 
 and then we order the $k$-subsets of $ S$
 in the lexicographic order with respect the  order of $S$),
 the $k$-weights with this order give
a vector in $\mathbb{R}^{ \sharp S \choose k}$. This vector is called 
$k${\bf -dissimilarity vector} of $({\cal G}, S)$.
Equivalently, if we don't fix any order, we can speak 
of the {\bf family of the $k$-weights}  of $({\cal G}, S)$.

We can wonder when a family of real numbers is the family of the $k$-weights
of some weighted graph and of some subset of the set of its vertices.
If $S$ is a finite set, $k \in \N$ and $k < \sharp S$,
we say that a family of positive real numbers
 $\{D_{I}\}_{I \in {S \choose k}}$  is {\bf p-graphlike} (resp. nn-graphlike, 
ip-graphlike)) if 
there exist a positive-weighted (resp. nonnegative-weighted, 
internal-positive-weighted)  graph ${\cal G}=(G,w)$ and a subset $S$ 
of the set of its vertices such that $ D_{I}({\cal G}) = D_{I}$  for any 
$ I $ $k$-subset of $ S$.
If the graph is a positive-weighted tree, we say that the family
 is {\bf p-treelike};
if the graph is a positive-weighted tree and $S \subset L(G)$,
we say that the family
 is {\bf p-l-treelike} (and analogously for nonnegative-weighted
or positive-internal-weighted trees).

The first contribute to the 
characterization of the graphlike families of numbers is due to Hakimi and 
Yau; in 1965, they observed that a family of positive real numbers
$\{D_{I}\}_{\{I \}\in {\{1,...,n\} \choose 2}}$
is p-graphlike if and only if the $D_I$ satisfy the 
triangle inequalities.

In the same years, also a criterion for a metric on
a finite set to be nn-treelike
%realized by a  nonnegative-weighted tree with leaf set $\{1,..., n\}$ 
was established, see \cite{B}, \cite{SimP}, \cite{St}: 

%For trees the main result in the case $k=2$ is due to Buneman:
%in 1971 he characterized the metrics on finite sets that are
%nn-treelike (a partial result 
%in this direction had already been obtained in \cite{SimP}):

\begin{thm} %{\bf (Buneman)} 
Let
$\{D_{I}\}_{I \in {\{1,...,n\} \choose 2}}$ be a set of positive real numbers 
satisfying the triangle inequalities.
It is p-treelike (or nn-l-treelike) 
 if and only if, for all $i,j,k,h  \in \{1,...,n\}$,
the maximum of $$\{D_{i,j} + D_{k,h},D_{i,k} + D_{j,h},D_{i,h} + D_{k,j}
 \}$$ is attained at least twice. 
\end{thm}

In terms of tropical geometry, the theorem above can be formulated by saying 
that 
the set of the $2$-dissimilarity vectors of weighted trees with $n$ leaves
and such that  the internal 
edges have  negative weights is the tropical Grassmanian $ {\cal G}_{2,n}$
(see \cite{SS2}).

For $k=2$ also  
the case of not necessarily nonnegative weights has been studied.
For any weighted graph ${\cal G}=(G,w)$ and for any $i,j \in V(G)$, we define
$D_{i,j}({\cal G})$ to be the minimum of $w(p)$ 
for $p$ a simple path joining $i$ and $j$. 
Again, we call such numbers ``$2$-weights''.
In 1972 Hakimi and Patrinos proved that a family of real numbers
$\{D_{I}\}_{\{I \}\in {\{1,...,n\} \choose 2}}$
  is always the family of the $2$-weights
of some weighted graph and some subset $\{1,...., n\}$ of its vertices.

In \cite{B-S}, Bandelt and Steel proved a result, analogous to
 Buneman's one, for general weighted trees, precisely they
 proved that, 
for any set of real numbers $\{D_{I}\}_{\{I \}\in 
{\{1,...,n\} \choose 2}}$,
 there exists a weighted tree ${\cal T}$ with leaves $1,...,n$
such that $ D_{I} ({\cal T})= D_{I}$  for any $I$ $2$-subset of
 $\{1,...,n\}$  if and only
 if, for any $a,b,c,d \in  \{1,...,n\}$,  we have that at least two among 
 $ D_{a,b} + D_{c,d},\;\;D_{a,c} + D_{b,d},\;\; D_{a,d} + D_{b,c}$
are equal.

For higher $k$ the literature is more recent.
In 2004, Pachter and Speyer proved the following theorem (see \cite{P-S}).

\begin{thm} {\bf (Pachter-Speyer)}. Let $ k ,n  \in \mathbb{N}$ with
$ n \geq 2k-1$ and $ k \geq 3$.  A positive-weighted tree
 ${\cal T}$ with leaves $1,...,n$ and no vertices of degree 2
is determined by the values $D_I({\cal T})$, where $ I $ varies in  
${\{1,...,n\} \choose k }$.
\end{thm}

In \cite{H-H-M-S}, the authors gave the following 
 characterization of the ip-l-treelike families of positive real numbers:

\begin{thm} {\bf  (Herrmann, Huber, Moulton, Spillner)}.
If $n \geq 2k$, 
a family  of positive real numbers 
$ \{D_{I}\}_{I \in {\{1,..., n\} \choose k}}$ is ip-l-treelike 
if and only if the restriction to every $2k$-subset of $\{1,...,n\}$ is 
ip-l-treelike.   
\end{thm}

Besides they studied when a family  of positive real numbers is ip-l-treelike
in the case $k=3$.

Finally, in \cite{Ru1} and \cite{Ru}, for any weighted  tree ${\cal T}$, for 
$k \geq 2$ and for any distinct leaves $ i_1,..., i_k$, the author defines 
  $D_{i_1,....,i_k}({\cal T})$ to be the sum of the lengths of the edges 
of the minimal  subtree joining $i_1$,....,$i_k$ and
gives an inductive
 characterization of the families of real numbers 
indexed by the subsets of $\{1,...,n\}$
 of  cardinality greater or equal than $2$, that
 are the families of  the multiweights of a tree with $n$ leaves and the 
set of its leaves.

Let $n ,k \in \N$ with $n >k$.
In this paper we  study the problem of the  characterization
of  the  families of positive real numbers, indexed by the $k$-subsets
 of an $n$-set, that are p-graphlike. 
%that is that are the family of the $k$-weights
%of some positve weighted graph and some $n$-subset of the set of the vertices.
As we have already said, the case $k=2$ has already been studied by Hakimi and 
Yau.
Here we examine the case $k=n-1$, both for trees (Section 3)  and graphs
(Section 4), 
see Theorems \ref{thm:trees}, \ref{thm:trees2}, \ref{thm:graphs}
and \ref{thm:morevertices}.

\section{Notation and a first remark }

\begin{nota}
$\bullet $ Let $ \mathbb{R}_{+} =\{x \in \mathbb{R} | \; x >0\}$.

$ \bullet $ For any $n \in \N $ with $ n \geq 1$, let $[n]= \{1,..., n\}$.

$ \bullet $ For any set $S$ and $k \in \mathbb{N}$,  let ${S \choose k}$
be the set of the $k$-subsets of $S$.

$ \bullet $ 
Throughout the paper, we will consider only simple finite graphs.
  For any 
$v,v' \in V(G)$, let $e(v,v') $ denote the edge joining $v$ and $v'$.

$\bullet $ For simplicity, the vertices of graphs
 will be often named with natural numbers. In the figures the names of the 
vertices will be written in bold calligrahic in order to avoid confusion with 
the weights.
\end{nota}

\begin{defin} \label{intest}
Let ${\cal G}=(G,w) $ be a positive-weighted graph. 
We say that a connected subgraph of $G$, $R$, realizes $ D_{i_1,...., i_k}({\cal G})$,
 or it is a {\bf subgraph realizing} $ D_{i_1,...., i_k}({\cal G})$ if 
$i_1,...,i_k$ are vertices of $R$ and 
$w(R)= D_{i_1,...., i_k}({\cal G})$. Observe that it is necessarily a tree, so we will 
call it also {\bf subtree realizing  $ D_{i_1,...., i_k}({\cal G})$}.

Let $S$ be a finite set. 
If $D_I$ for $I \in {S \choose k}$   are positive real numbers 
and there exist a nonnegative weighted graph $ {\cal G} =(G,w)$  
and a subset $S$  of $V(G)$, such that 
$D_I ({\cal G}) = D_I$ for any 
$I \in {S \choose k}$, then we say that $({\cal G}, S) $ realizes the family 
 $\{D_{I}\}_{I \in {S \choose k}}$ and, following \cite{H-Y}, 
 we call  the vertices  in $S$   {\bf external vertices} and 
the other vertices {\bf internal vertices}.
\end{defin}

\begin{rem}\label{rem:triangleinequality}
Let ${\cal G}= (G,w) $ be a positive-weighted graph; then, for any 
$I,J,K \subset V(G) $ with $ J \cap K \neq \emptyset$ and $ J \cup 
K \supset I$,
we have the ``triangle inequality''
\begin{equation} \label{eq:triangle}
D_I ({\cal G}) \leq D_J ({\cal G}) + D_K ({\cal G}). 
\end{equation}
\end{rem}

\begin{proof}
Let $G'$ and $G''$ be two subgraphs of $G$   
realizing respectively $D_J ({\cal G}) $ and $D_K ({\cal G}) $: 
the union of these subgraphs is a connected subgraph, $G'''$, 
whose set of vertices contain $J$ and $K$ and then $I$.
So we have: $$D_I ({\cal G}) \leq w(G''') 
\leq w(G') + w(G'') = D_J ({\cal G}) + D_K ({\cal G}). $$
\end{proof}

\section{$(n-1)$-dissimilarity vectors of trees with $n$ vertices}

In this section we want to examine when a family of positive real 
numbers  $\{D_I\}_{I \in {[n] \choose n-1 }}$ is treelike.

\begin{nota} Let $n \in \mathbb{N},\;\; n \geq 3$. 
Given a family of real numbers $\{D_I\}_{I \in {[n] \choose n-1 }}$,
we denote $D_{1,..., \hat{i},..., n}$ by
$D_{\hat{i}}$ for any $i \in [n]$.
\end{nota}

\begin{thm}\label{thm:trees}
Let $n \in \mathbb{N},\;\; n \geq 3$ and 
let  $\{D_I\}_{I \in {[n] \choose n-1 }}$ be a family of positive real 
numbers.
%Let $D_{\hat{i}} \in \mathbb{R}_{+}$ for $i=1,..., n$.

(a) There exists a positive-weighted tree ${\cal T} =(T,w)$ 
with at least $n$  {\em vertices}, $1,...,n$,  and such that  
$D_{\hat{i}} ({\cal T}) = D_{\hat{i}}$ for any  $i=1,..., n$
if and only if 
 \begin{equation} \label{disn-2}
(n-2)D_{\hat{i}} \leq \sum_{j \in [n]-\{ i\} } D_{\hat{j}}
\end{equation} for any  $  i \in [n] $
and at most one of the inequalities  (\ref{disn-2}) is an equality.

(b) There exists  a positive-weighted tree ${\cal T} =(T,w)$ 
 with at least $n$ {\em leaves}, $1,...,n$,  and such that  $D_{\hat{i}} ({\cal T}) = D_{\hat{i}}
$ for any  $i \in [n]$
if and only if 
 \begin{equation} \label{strictn-2}
(n-2)D_{\hat{i}} < \sum_{j \in [n] -\{i\} } D_{\hat{j}}
\end{equation} for any  $  i \in [n] $.
\end{thm}

\begin{proof} (a) $\Rightarrow$ Let ${\cal T}= (T,w)$ be a positive weighted 
 tree and let $ [n] \subset V(T)$. 
We want to show that for any $ k \in [n]$  
\begin{equation} \label{n-2hati}
(n-2) D_{\hat{k}}({\cal T}) \leq \sum_{j \in [n] -\{k\}} D_{\hat{j}}
({\cal T}).
\end{equation}
 Let $G_{1,...,n}$ be a subtree of $T$ 
 realizing $D_{1,..., n}({\cal T})$; obviously 
it is a tree with set of  leaves contained in $[n]$. Let 
  $G_{\hat{k}}$ be a subtree of $T$ 
 realizing $D_{\hat{k}}({\cal T})$; obviously 
it is a tree with set of  leaves contained in 
$\{1,...,\hat{k}, ..., n\}$.
Observe that $ G_{\hat{k}}$ is a subgraph of 
 $ G_{1,..., n}$ (in fact, for every $r,s \in \{1,...,\hat{k}, ..., n\}$, 
 the path between $r$ and $s$ in  $ G_{\hat{k}}$ must coincide with
the path between $r$ and $s$ in $G_{1,...,n}$, 
 since $T$ is a tree).
For every $k=1,..., n$, let $a_k$ be the weight of the subgraph 
 $ G_{1,..., n}-  G_{\hat{k}}$. The inequality (\ref{n-2hati})
becomes
$$ (n-2)(D_{1,...,n} ({\cal T}) -a_k) \leq \sum_{j \in [n] -\{k\}} 
(D_{1,...,n} ({\cal T}) -a_j),
$$ 
which is equivalent to 
$$  \sum_{ j \in [n] -\{k\} } a_j \leq D_{1,...,n}({\cal T}) + (n-2) a_k,$$ 
which is true, since, obviously, 
$  \sum_{ j \in [n] -\{k\} } a_j \leq D_{1,...,n} ({\cal T})$.

Now we want to prove that at most one of the inequalities  (\ref{disn-2}) is an equality.
Suppose
$$(n-2) D_{\hat{k}}({\cal T}) = \sum_{ j \in [n] -\{k\} } D_{\hat{j}} ({\cal T}).$$ 
Then $$  \sum_{ j \in [n] -\{k\} } a_j = D_{1,...,n} ({\cal T}) + (n-2) a_k.$$ 
Since
$  \sum_{ j \in [n] -\{k\} } a_j \leq D_{1,...,n}({\cal T})$, 
we get $a_k=0$ and $$  D_{1,...,n} ({\cal T})= \sum_{ j \in [n] -\{k\} } a_j .$$
Thus $G_{1,..., n}$ is a star tree with leaves $1,...., \hat{k}, ...., n$ 
and center $k$ and $a_j=0$ if and only if $j=k$. It is easy to check that in
 this case only one of the inequalities  (\ref{disn-2}) is an equality.

 $\Leftarrow$ 
We consider two cases: the case where
all the inequalities  (\ref{disn-2}) are strict and 
the case where exactly 
one of the inequalities  (\ref{disn-2}) is an equality.

$\bullet$
First let us suppose that all the inequalities  (\ref{disn-2}) are strict.
Let $ {\cal T}$ be the star tree with $[n]$ as set of leaves and center 
$n+1$ and, for $k=1,...,n$, let   $$w( e(n+1,k))=
 \frac{ \sum_{ j \in [n] -\{k\} } D_{\hat{j}}-
(n-2) D_{\hat{k}}}{n-1} . $$ 
For any $k=1,....,n$, we have that 
$D_{\hat{k}} ({\cal T})$ is equal to the sum of the
 weights of all the edges but the edge  $ e(n+1,k)$. Therefore
$$D_{\hat{k}} ({\cal T}) = \frac{1}{n-1} \left[ \sum_{h \in [n] -\{k\}}
\;\; \left( \sum_{ j \in [n] -\{h\} } D_{\hat{j}}-
(n-2) D_{\hat{h}} \right)
\right]= $$ 
$$ = \frac{1}{n-1} \left[ \sum_{ h \in [n] -\{k\} }
\;\; \left(\sum_{j \in [n]} D_{\hat{j}}-
(n-1) D_{\hat{h}} \right) \right]=$$ 
$$ = \frac{1}{n-1} 
\left[ 
(n-1) \left( \sum_{j \in [n]} D_{\hat{j}} \right)
\;\;-
\;\; (n-1)
 \left( \sum_{h \in [n] -\{k\} }  D_{\hat{h}} \right) 
\right]
= D_{\hat{k}}. $$ 

$\bullet$
Now let us suppose that exactly 
one of the inequalities  (\ref{disn-2}) is an equality, precisely
\begin{equation} \label{eqr} 
(n-2) D_{\hat{r}} = \sum_{ j \in [n] -\{r\} } D_{\hat{j}}.
\end{equation}
Let $ {\cal T}$ be the star tree with $\{1,..., \hat{r},..., n\}$ as 
set of leaves and center 
$r$ and, for $k \in [n] -\{r\}$, let   $$
 w(e(r,k))= \frac{ \sum_{ j \in [n] -\{k\} } D_{\hat{j}}-
(n-2) D_{\hat{k}}}{n-1} . $$ 
We want to show that, for any $k=1,..., n$, we have that 
$D_{\hat{k}} ({\cal T}) =D_{\hat{k}} $. 
Let $k \neq r$.  
Thus  $D_{\hat{k}} ({\cal T})$ is equal to the sum of the
 weights of all the edges but the edge  $e(r,k)$.
Therefore
$$D_{\hat{k}} ({\cal T}) = \frac{1}{n-1} \left[ \sum_{h \in [n] -\{r,k\}}
\;\; \left(\sum_{j \in [n] -\{h\}} D_{\hat{j}}-
(n-2) D_{\hat{h}} \right)
\right]= $$ 
$$ = \frac{1}{n-1} \left[ \sum_{ h \in [n] -\{r,k\} }
\;\; \left( \sum_{j \in [n]} D_{\hat{j}}-
(n-1) D_{\hat{h}} \right) \right]=$$ 
$$ = \frac{1}{n-1} 
\left[ 
(n-2) \left( \sum_{j \in [n]} D_{\hat{j}} \right)
\;\;-
\;\; (n-1)
 \left( \sum_{h \in [n]}  D_{\hat{h}} -D_{\hat{r}} - D_{\hat{k}}\right) 
\right]= $$  $$ =
\frac{1}{n-1} 
\left[ 
 -\sum_{j \in [n]} D_{\hat{j}} 
+(n-1)D_{\hat{r}} + (n-1)D_{\hat{k}}
\right]= D_{\hat{k}}, $$ 
where the last equality holds because, by (\ref{eqr}), 
we have that $(n-1)D_{\hat{r}} =  \sum_{j=1,..., n} D_{\hat{j}} $.
Besides, for any $k \neq r$,
$$ D_{\hat{r}} ({\cal T}) =D_{\hat{k}} ({\cal T})  + w(e(r,k))= 
D_{\hat{k}}  + w(e(r,k))= D_{\hat{r}} , $$
    where the last equality holds again by (\ref{eqr}).

%$$D_{\hat{r}} ({\cal T}) = \frac{1}{n-1} \left[ \sum_{h \in [n] -\{r\}}
%\;\; \left(\sum_{ j \in [n] -\{h\} } D_{\hat{j}}-
%(n-2) D_{\hat{h}} \right) \right]= $$ 
%$$ = \frac{1}{n-1} \left[ \sum_{ h \in [n] -\{r\}}
%\;\; \left(\sum_{j \in [n]} D_{\hat{j}}-(n-1) D_{\hat{h}} \right) \right]=$$ 
%$$ = \frac{1}{n-1} \left[ 
%(n-1) \left( \sum_{j \in [n]} D_{\hat{j}} \right)\;\;-
%\;\; (n-1) \left( \sum_{ h \in [n] -\{r\} }  D_{\hat{h}} \right) \right]
%= D_{\hat{r}} $$ 

(b) $\Rightarrow $ We can argue as in 
the analogous implication in (a). In this 
case, since $1,...,n $ are leaves, all the $a_i $ are nonzero, so we get 
the strict inequalities (\ref{strictn-2}).

$ \Leftarrow $ 
Let $ {\cal T}$ be the star tree with $\{1,...., n\}$ as set of leaves and center 
$n+1$ and, for $k \in [n]$, let $$ w(e(n+1,k ))
= \frac{ \sum_{j \in [n] -\{k\}} D_{\hat{j}}-
(n-2) D_{\hat{k}}}{n-1}. $$
We can easily prove that $ D_{\hat{k}} ({\cal T}) = D_{\hat{k}}$ for any 
$k=1,...,n$.

\end{proof}

Now we examine the case where there are no internal vertices (see Definition 
\ref{intest}).

\begin{thm}\label{thm:trees2}
Let $n \in \mathbb{N},\;\; n \geq 3$  and 
let  $\{D_I\}_{I \in {[n] \choose n-1 }}$ be a family of positive real 
numbers.
%Let $D_{\hat{i}}  \in \mathbb{R}_{+}$ for $i \in [n] $. 

There exists  a positive-weighted tree ${\cal T} =(T,w)$ 
 with exactly $n$   vertices, $1,..., n$,
 and such that  $D_{\hat{i}} ({\cal T}) = 
D_{\hat{i}}$ for any  $i \in [n]$
if and only if the following three conditions hold:

(i)  \begin{equation} \label{disn-2bis}
(n-2)D_{\hat{i}} \leq \sum_{j \neq i} D_{\hat{j}}
\end{equation} for any  $ i   \in [n] $ 
and at most one of the inequalities  (\ref{disn-2bis}) is an equality.

(ii)  either one of the inequalities 
(\ref{disn-2bis}) is an equality or the maximum in $ 
\{D_{\hat{i}}\}_{i \in [n]} $ is acheived at least twice.

(iii)  the maximum in $ \{D_{\hat{i}}\}_{i  \in [n] } $
 is acheived at most $n-2$ times.
\end{thm}

\begin{proof}
$\Rightarrow$ (i) The inequality follows from the analogous implication of 
Theorem \ref{thm:trees}.

(ii), (iii) 
Since $n \geq 3$, at least one of the vertices of $T$
must have degree greater or equal than $2$. 
Besides observe that, if a vertex $k $ of $T$ 
has degree greater or equal than $2$, then $ D_{\hat{k}}({\cal T}) 
=   D_{1,...., n} ({\cal T})$, which is greater or equal than  
$D_{\hat{j}} ({\cal T})$ for every $j=1,..., n$. So 
$ D_{\hat{k}}({\cal T}) $ is the maximum of $ 
\{D_{\hat{i}}({\cal T})\}_{i \in [n] } $.

Thus, if in $T$ there are at least two vertices
 of degree greater or equal than $2$, then
the maximum in    $ \{D_{\hat{i}}({\cal T})\}_{i \in [n]} $ is acheived twice.

If in $T$ there is only one vertex, $k$, of degree 
greater or equal than $2$, then $T$ is a star tree with center $k$ 
 and we can check easily that only one of the inequalities (\ref{disn-2bis})
 is an equality. 

Therefore we have proved (ii).

To prove (iii),
observe that in $T$ there are at least two leaves, $r$ and $s$. Since they 
are leaves, we have:
$$ D_{\hat{r }} ({\cal T}), D_{\hat{s}} ({\cal T}) < D_{1,..., n} ({\cal T})= 
D_{\hat{k}}({\cal T}) $$ for any $k$ vertex of degree at least $2$.
So the maximum in    $ \{D_{\hat{i}} ({\cal T}) \}_{i \in [n]} $ is acheived at most 
$n-2$ times and we have proved (iii).

$\Leftarrow$ Suppose (i),(ii), and (iii) hold. 
\begin{itemize}
\item If at least one (and then, by assumption, exactly one)  
of the inequalities (\ref{disn-2bis})  is an equality,
let us say $(n-2) D_{\hat{r}} = \sum_{j=1,..., n,\; j\neq r} D_{\hat{j}}$, 
let $ {\cal T}$ be the star tree with $\{1,..., \hat{r},..., n\}$ as 
set of leaves and center 
$r$ and, for $k=1,...,n$, $k \neq r$, let $$w(e( r,k))=
 \frac{ \sum_{j=1,..., n,\; j\neq k} D_{\hat{j}}-
(n-2) D_{\hat{k}}}{n-1} . $$ 
We can easily prove that $ D_{\hat{k}} ({\cal T}) = D_{\hat{k}}$ for any 
$k=1,...,n$.

\item If all the inequalities (\ref{disn-2bis}) are strict, then, 
by assumption,
 the maximum in $ \{D_I\}_{I \in {[n] \choose n-1}} $
 is acheived at least twice.
So we can suppose 
$$ D_{\hat{1}} = .....= D_{\hat{h}} > D_{\hat{h+1}}, ...., D_{\hat{n}}$$ 
for some  $h\geq 2$. 
By (iii) we have that  $n-h\geq 2$.
Let us divide the set $\{h+1,...., n\}$ into two nonempty subsets:
$\{h+1,....,h+s\}$ and $\{h+s +1,..., n\}$.

Let ${\cal T}= (T,w) $ be the weighted tree with leaves $h+1,...,n$ in 
Fig. \ref{treenvert} such that

\begin{figure}[!h]
\centering
\input{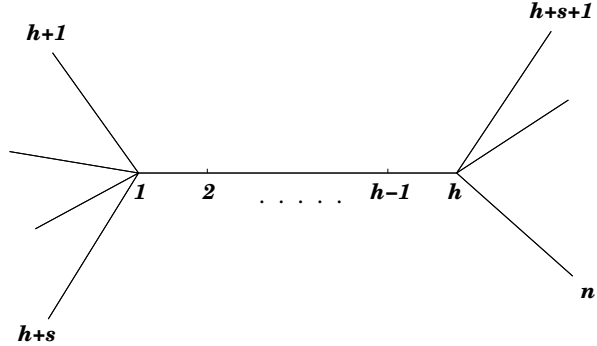}
\caption{ Tree in the case all the inequalities are strict \label{treenvert}}
\end{figure}

\medskip

$w(e(k,1) ) =   D_{\hat{1}} - D_{\hat{k}}$ for $k=h+1,..., h+s$, 

$w(e(k,h) ) =   D_{\hat{1}} - D_{\hat{k}}$ for $k=h+s+1,...,n$, 

$w(e(k,k+1)) =\frac{\sum_{j \geq h+1 }  D_{\hat{j}} -(n-h-1)  
D_{\hat{1}}}{h-1} $  for $k=1,...,h-1$.

\medskip

So the path between $1$ and $h$ has weight 
$\sum_{j \geq h+1 }  D_{\hat{j}} -(n-h-1)  D_{\hat{1}}$. 
Observe that the weights of the edges of ${\cal T}$ are positive, in fact 
$D_{\hat{1}} > D_{\hat{k}}$ for $k=h+1,..., n$ and 
$$ \sum_{j \geq h+1 }  D_{\hat{j}} -(n-h-1)  D_{\hat{1}} =
\sum_{j \geq 2 }  D_{\hat{j}} -(n-2)  D_{\hat{1}},$$ 
which is positive because we are in the case where    
all the inequalities (\ref{disn-2bis}) are strict.

Let $k \in \{1,..., h\}$. Then $ D_{\hat{k}} ({\cal T})$ is the sum 
of the weights of all the edges of ${\cal T}$, thus:
$$ D_{\hat{k}} ({\cal T})= \sum_{j =h+1,...,n} ( D_{\hat{1}} - D_{\hat{j}})
+  \sum_{j =h+1,...,n} D_{\hat{j}} - (n-h-1) D_{\hat{1}}  = D_{\hat{1}} =
D_{\hat{k}}.$$ 

Let $k \in \{h+1,..., n\}$. Then $ D_{\hat{k}} ({\cal T})$ is the sum 
of the weights of all the edges of ${\cal T}$ but the edge $e(k,1)$ if 
$k \in  \{h+1,..., h+s\}$ or the edge $e(k,h)$ if  $k \in  \{h+s+1,..., n\}$.
Hence $$ D_{\hat{k}} ({\cal T})=  D_{\hat{1}} - ( D_{\hat{1}} - D_{\hat{k}})  =
D_{\hat{k}}.$$

\end{itemize}

\end{proof}

\section{$(n-1)$-dissimilarity vectors of graphs with $n$ vertices}

In this section we want to examine when 
a family of positive real numbers  $\{D_I\}_{I \in {[n] \choose n-1 }}$
is graphlike.
 First we consider the case of graphs  with exactly $n$ vertices.

\begin{thm}\label{thm:graphs} Let $n \in \mathbb{N},\;\; n \geq 3$.
Let $\{D_I\}_{I \in {[n] \choose n-1 }}$
be a family of positive real numbers. 

There exists a positive weighted graph ${\cal G}=(G,w)$ with exactly $n$ vertices, 
$1,..., n$, and with 
$D_{\hat{i}}({\cal G}) =D_{\hat{i}}$ for any $i=1,..., n$
 if and only if the following two conditions hold:

(i)
\begin{equation} \label{n-2} 
(n-2)D_{\hat{i}} \leq \sum_{j=1,..., n, \;\; j \neq i}D_{\hat{j}}
\end{equation}
  
for any   $i \in [n]$,

(ii)  if the maximum in $\{D_{\hat{i}} 
\}_{i \in [n]} $ is acheived at least twice, the
inequalities (\ref{n-2}) are strict. 
\end{thm}

\begin{rem}\label{n-2implicak}
 Condition (i) implies that, for any $k \in \N$ with $ 1 \leq k \leq n-2$,
\begin{equation} \label{primon-22}
 k D_{\hat{i}} \leq \sum_{j \in \{i_1,...., i_{k+1}\}} 
D_{\hat{j}}
\end{equation}
for any distinct $ i,i_1,..., i_{k+1} \in [n] $, in particular 
condition (i) implies the triangle inequalities.
\end{rem}

\begin{proof}
Obviously we can suppose that, in (\ref{primon-22}),
 $D_{\hat{i}} = max_{j \in [n]} D_{\hat{j}}$.
Therefore
 $$(n-k-2)D_{\hat{i}} \geq \sum_{j \in [n] -\{i,i_1,...., i_{k+1}\}} 
D_{\hat{j}}.$$
The inequality above and condition (i)
imply at once the inequality  (\ref{primon-22}).
\end{proof}

{\em Proof of Theorem \ref{thm:graphs}}
$ \Rightarrow $ We can suppose $i=1 $ in (\ref{n-2}), so 
the inequality we have to prove becomes 
 $$(n-2) D_{\hat{1}} \leq \sum_{j=2,..., n}  D_{\hat{j}}.$$

Let $H$ be a subtree of $G$  realizing $D_{\hat{n}}$. 
Observe that we can construct an injective map 
$$ \varepsilon: V(H) -\{1\} \longrightarrow E(H)$$ sending a vertex $v$ 
to  an edge 
incident with $v$ (consider $H$ as rooted tree with $1$ as root and send any 
leaf to the unique incident edge, then delete these edges and repeat the 
procedure until you arrive at the root).  

Observe that $$ D_{\hat{1} } \leq D_{ \hat{i}}
+ w(\varepsilon(i)) $$ for $i=2,....., n-1$, in fact: let
  $R_i$ be the subgraph 
given by the union of a subtree  realizing $D_{\hat{i}}$ and
 of $ \varepsilon(i)$; the subgraph $R_i$  is connected (because the ends of 
$\varepsilon(i) $ are
$i$ and a vertex among $1,...., n$ different from $i$, thus a vertex of the 
subtree realizing $D_{ \hat{i}}$); besides 
  $2,..., n \in V(R_i) $; hence $ D_{\hat{1}} \leq w(R_i) = 
D_{ \hat{i}}+ w(\varepsilon(i))$.
Therefore  $$(n-2) D_{\hat{1} } \leq \sum_{i=2,..., n-1} 
\left(D_{ \hat{i}}
+ w(\varepsilon(i))\right) \leq $$ 
$$ \leq \sum_{i=2,..., n-1}  D_{  \hat{i}}+ \sum_{i=2,..., n-1} 
 w(\varepsilon(i)) \;  \leq \sum_{i=2,..., n-1} D_{ \hat{i}}
+ D_{\hat{n}} = \sum_{i=2,..., n}  
D_{\hat{i}}.$$

Suppose now that the maximum of $\{D_{\hat{i}}\}_{i \in [n]}$ is achieved 
at least twice. 
We want to show that the inequalities (\ref{n-2}) are strict.

We can suppose  $D_{\hat{1}}= D_{\hat{2}} \geq D_{\hat{j}}$ 
for $j=3,....,n$. Obviously,  
to prove that the  inequalities  (\ref{n-2}) are strict, it suffices to prove 
that $$(n-2) D_{\hat{1} } < \sum_{i=2,..., n}  D_{\hat{i}}.$$
We have already proved that 
$ D_{\hat{1} } \leq D_{ \hat{i}}
+ w(\varepsilon(i)) $ for $i=2,....., n-1$, in particular 
for $i=3,....., n-1$; besides we know that
$ D_{\hat{1} } = D_{ \hat{2}}$.
Thus we get: 
 $$(n-2) D_{\hat{1} } %= D_{\hat{2}} + (n-3) D_{\hat{1} } 
\; \leq \; D_{\hat{2}} +\sum_{i=3,..., n-1} 
\left(D_{ \hat{i}}
+ w(\varepsilon(i))\right) \; = \; \sum_{i=2,..., n-1} 
 D_{  \hat{i}}+ \sum_{i=3,..., n-1 }  
 w(\varepsilon(i)) \;<  $$  $$ 
 <  \; \sum_{i=2,..., n-1} 
 D_{  \hat{i}}+ \sum_{i=2,..., n-1 }  
 w(\varepsilon(i)) \; 
\leq \;
\sum_{i=2,..., n-1} D_{ \hat{i}}
+ D_{\hat{n}}=  \sum_{i=2,..., n}  
D_{\hat{i}}.$$

$\Leftarrow $ 
Let $n=3$. Let ${\cal G}$ be the complete graph  with vertices $1,2,3$ and 
 weights $w(e(1,2)) = D_{\hat{3}}$, $w(e(1,3)) = 
D_{\hat{2}}$,  $w(e(2,3)) = D_{\hat{1}}$. Obviously $D_{\hat{i}} ({\cal G})=
D_{\hat{i}}$ for $i=1,2,3$.

Let $n\geq 4$. We consider two cases: 

1) the maximum in $\{D_{\hat{i}}\}_{i \in [n] }$ is acheived at least twice,

2) the maximum in $\{D_{\hat{i}}\}_{i \in [n] }$ is acheived only once.

\bigskip

1)  Suppose $ D_{\hat{1}}= ....= D_{\hat{k}} > D_{\hat{k+1}}, ..., D_{\hat{n}}
$ for some $k \geq 2$.
Let $$a = \frac{D_{\hat{k+1}}+....+ D_{\hat{n}} - (n-k-1) D_{\hat{1}}}{n-2}$$ 
and, for any $ i=k+1, ...., n$, let 
$$x_i=  \frac{D_{\hat{k+1}}+....+ D_{\hat{n}} }{n-2} +\frac{k-1}{n-2} 
D_{\hat{1}} - D_{\hat{i}}.$$ 

We can easily prove that $$ a \leq x_i$$ 
for any $i=k+1,..., n$. Besides observe that $a$ 
is positive by assumptions (i) and (ii) (and then also the $ x_i$ are positive).

Let ${\cal G}$ be the weighted graph defined in the following way:
consider the complete graph with vertices $1,...,k$ 
and weights of the edges equal to $a$ and then, for any $i=k+1,..., n$, 
 draw an edge joining $i$ and $1$ and an edge joining $i$ and $k$, both
with weight $x_i$. 

Since $a \leq x_i$ for any $i=k+1,...,n$, we get: 
$$ D_{\hat{1}} ({\cal G})=.....= D_{\hat{k}} ({\cal G}) =
 (k-2)a + x_{k+1}+.....+ x_n =  D_{\hat{1}}=.....= D_{\hat{k}}   $$ 
and, for any $i=k+1,...,n$, we have:
$$ D_{\hat{i}} ({\cal G})=
 (k-1)a + \sum_{j=k+1,..., n,\; j\neq i}
x_{j} =  D_{\hat{i}} . $$

2) We prove the statement by induction on $n$. Precisely 
we prove,  by induction on $n$ (with $n=4$ as base case), that 
if (i) holds  and the maximum in $\{D_{\hat{i}}\}_{i \in [n]}$ 
is acheived only once, then there exists a weighted graph ${\cal G}= 
(G,w)$ with exactly $n$ vertices such that: 

\begin{itemize}
\item[-]
$D_{\hat{j}} = D_{\hat{j}} ({\cal G})$  for any $j=1,...,n$

\item[-] if $ D_{\hat{i}}({\cal G})$ is the 
maximum in $\{D_{\hat{j}}\}_{j \in [n]}$ ,  then 
any subgraph realizing $ D_{\hat{i}} ({\cal G})$
has necessarily $i$ as vertex, so in particular $D_{1,...,n} ({\cal G})
= D_{ \hat{i}}({\cal G})$.
\end{itemize}

  Let $n=4$. Suppose that $D_{\hat{1}} > D_{\hat{2}},   D_{\hat{3}},   
D_{\hat{4}}$. 
Without loss of generality we can also suppose that  $D_{\hat{3}} \geq
   D_{\hat{2}}$.
Let ${\cal G}$ be the weighted graph shown in Figure \ref{fig:graph2dis}.

\begin{figure}[!h]
\centering
\input{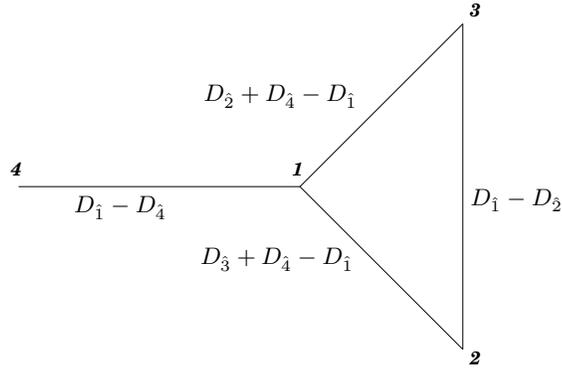}
\caption{Graph with $3$-dissimilarity vector $
(D_{\hat{1}}, D_{\hat{2}},   D_{\hat{3}},  D_{\hat{4}})$ \label{fig:graph2dis}}
\end{figure}

Observe that it is positive weighted, in fact: 
obviously $ D_{\hat{1}} - D_{\hat{4}} >0$ and
$ D_{\hat{1}} - D_{\hat{2}} >0$; besides
$$ D_{\hat{3}} + D_{\hat{4}} -  D_{\hat{1}} >0,$$ because 
$2  D_{\hat{1}} \leq D_{\hat{2}} + D_{\hat{3}}  + D_{\hat{4}}$ by 
(\ref{n-2}) and $D_{\hat{1}} > D_{\hat{2}}$; analogously 
$$ D_{\hat{2}} + D_{\hat{4}} -  D_{\hat{1}} >0.$$

Furthermore we can see easily that:

$ w(e(2,3)) \leq w(e(1,2))$,

$ w(e(1,3)) \leq w(e(1,2))$,

$w(e(1,3)) + w(e(2,3)) \geq w(e(1,2)) $, 

$w(e(1,2)) + w(e(2,3)) \geq w(e(1,3)) $.

Thus we get:

$D_{\hat{1}}({\cal G})= w(e(1,4)) + w(e(1,3)) + w(e( 2,3)) = D_{\hat{1}}$,
 
$D_{\hat{2}}({\cal G})= w(e(1,4)) + w(e(1,3)) =  D_{\hat{2}}$,

$D_{\hat{3}}({\cal G})= w(e(1,4)) + w(e(1,2)) =  D_{\hat{3}}$,

$D_{\hat{4}}({\cal G})= w(e(1,3)) + w(e(2,3)) =  D_{\hat{4}}$.

%$D_{\hat{i}}({\cal G}) = D_{{\hat{i}}}$  for $i=1,2,3,4$.

Now we want to prove the induction step. Let $ n \geq 5 $. 
Without loss of generality we can suppose that
$$D_{\hat{1}}  > D_{\hat{n}} \geq D_{\hat{j}}$$
for  any $j=2,...,n-1$.

Let $x = D_{\hat{1}} - D_{\hat{n}}$ and define
$ \tilde{D}_{1,...,\hat{i},...., , n-1} $ ($ 
\tilde{D}_{\hat{i}}$ for short) for $i=1,..., n-1$ in the following way:
 
$$ \tilde{D}_{\hat{i}} = 
\left\{ 
\begin{array}{ll}
D_{\hat{n}} & \text{ for } i=1 \\
D_{\hat{i}} -x & \text{ for } i=2,...,n-1 
\end{array}
\right.
 $$

Observe that the $\tilde{D}_{\hat{i}} $
are positive, in fact the inequality 
$$ (n-2) D_{ \hat{1}} \leq  \sum_{j=2,..., n } D_{\hat{j}} $$
(which follows from (i)) and the inequalities
$$ D_{\hat{1}} > D_{ \hat{j}}$$ for any $j= 2,...,n $ (in particular for 
$j \neq 1,i,n$) imply that 
$$  D_{\hat{1}} < D_{\hat{n}} + D_{\hat{i}}. $$

Observe also that $ \tilde{D}_{\hat{1}} >\tilde{D}_{\hat{i}} $
for $i =2, ..., n-1$. Therefore also in the set 
$ \{\tilde{D}_{\hat{i}}\}_{i \in [n-1]} $, the maximum is acheived only once
(by $\tilde{D}_{\hat{1}} $).
Besides the $\tilde{D}_{\hat{i}} $ for  $i =1,...,n-1$
satisfy (\ref{n-2}) with $n-1$ instead of $n$, in fact:

obviously it suffices to prove (\ref{n-2}) when the first member is 
the maximum, 
that is $ (n-2) \tilde{D}_{\hat{1}} $, so it suffices to prove that 
$$ (n-3)  \tilde{D}_{\hat{1}} \leq 
\sum_{j=2,..., n-1} \tilde{D}_{\hat{j}},$$
which is equivalent to 
$$ (n-3)  D_{\hat{n}} \leq 
\sum_{j=2,..., n-1} D_{\hat{j}} - (n-2) D_{\hat{1}} + (n-2) D_{\hat{n}},$$
that is 
$$ (n-2)  D_{\hat{1}} \leq 
\sum_{j=2,..., n-1} D_{\hat{j}} + D_{\hat{n}},$$
which follows from (\ref{n-2}).

Therefore, by induction assumption, there exists a weighted 
 graph $\tilde{{\cal G}}= 
(\tilde{G}, \tilde{w}) $ with 
 vertices $1,..., n-1$ such that:
 $D_{\hat{i}}(\tilde{{\cal G}}) = \tilde{D}_{\hat{i}}$ for any $ i=1,...,n-1
$ and any subgraph realizing 
$ D_{\hat{1}} (\tilde{{\cal G}}) $
(which is the maximum of the 
$D_{\hat{i}} (\tilde{{\cal G}})$) has $1$ as vertex (so 
$ D_{\hat{1}} (\tilde{{\cal G}}) $ is equal to $ 
D_{1,...,n-1} (\tilde{{\cal G}})$).
Let ${\cal G}$ be the graph obtained from $\tilde{{\cal G}}$ 
by adding an edge $E$ incident with
$1$ with weight $x$; call $n$ the other end of $E$.

\begin{figure}[htbp]
\centering
\input{Gtilde.pstex_t}
\caption{${\cal G}$ and  $\tilde{{\cal G}}$. \label{fig:Gtilde}}
\end{figure}

Observe that 
any subgraph realizing $ D_{\hat{1}}({\cal G})$ 
has $1$ as vertex by the construction of ${\cal G}$.
We want to show that $D_{\hat{i}}({\cal G}) =D_{\hat{i}}$ for any $i=1
,..., n$.
Obviously, 
$$ 
D_{\hat{n}}({\cal G}) =
D_{1, ....,n-1} ({\tilde{\cal G}}) =  
D_{\hat{1}} ({\tilde{\cal G}}) = 
\tilde{D}_{\hat{1}}= 
D_{\hat{n}},
$$ 
$$ 
D_{\hat{1}}({\cal G}) = x+
D_{1, ....,n-1} ({\tilde{\cal G}}) =  
x+ D_{\hat{1}} ({\tilde{\cal G}}) = 
x+ \tilde{D}_{\hat{1}}= 
D_{\hat{1}}.$$ 
Besides, for $i=2,...,n-1$, 
$$ D_{\hat{i}}({\cal G}) = x + D_{\hat{i}} ({\tilde{\cal G}}) = 
x + \tilde{D}_{\hat{i}}= x+D_{\hat{i}}-x= D_{\hat{i} }.
$$

\hfill \framebox(7,7)

%\end{proof}

\bigskip

The case of graphs with internal vertices is more difficult: 
in this case condition~(\ref{n-2}) is not necessary: see 
\figurename~\ref{fig:examplenon2} for an example of a graph with four external vertices 
such that the $3$-dissimilarity vector 
doesn't satisfy  condition~(\ref{n-2}). Here we study the case with internal
 vertices only in the case $n=4$. It seems difficult
 to generalize the result to $n$ greater than $4$.

\begin{figure}[htbp]
\centering
\input{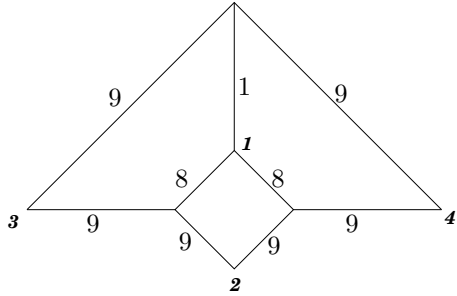}
\caption{Example of weighted graph such that 
 $2D_{\hat{1}}>D_{\hat{4}}+D_{\hat{3}}+D_{\hat{2}}$. \label{fig:examplenon2}}
\end{figure}

\begin{thm}\label{thm:morevertices}
Let $D_{1,2,3},D_{1,2,4},D_{1,3,4},D_{2,3,4} \in 
\mathbb{R}_{+}$; as usual we will denote them respectively by 
$D_{\hat{4}}, D_{\hat{3}} , D_{\hat{2}}, D_{\hat{1}}$. 
There exists a weighted graph ${\cal G}= (G,w)$ 
with $[4]$ subset of $V(G)$ 
such that $ D_{\hat{i}}({\cal G}) =D_{\hat{i}}$ for $i=1,2,3,4$
if and only if 

(i) $$
5D_{\hat{t}} \leq 3D_{\hat{k}}+ 3D_{\hat{j}} + 2D_{\hat{i}} 
$$
for any distinct  $i,j,k,t \in [4]$,

(ii) $$ D_{\hat{i}} <  D_{\hat{k}}+ D_{\hat{j}}$$ 
for any distinct  $i,j,k\in [4]$.
\end{thm}

%Let us make a remark before proving the theorem.

\begin{rem}
Condition~(i) implies the triangle inequalities (that is
the inequalities $D_{\hat{i}} \leq D_{\hat{j}} + D_{\hat{k}}$ for any $i,j,k$ 
distinct in $[4]$), but not the strict triangle inequalities.
\end{rem}
\begin{proof}
Condition~(i) implies 
 $$\, D_{\hat{t}} \leq \frac{3}{5}D_{\hat{k}} + \frac{3}{5}D_{\hat{j}} + \frac{2}{5}D_{\hat{i}} \,$$ for any distinct $i,j,k,t \in [4],$ so we get:
$$5D_{\hat{t}} \leq 3D_{\hat{k}} + 3D_{\hat{j}} + 2D_{\hat{i}} \leq \quad \quad \quad \quad \quad$$
$$\quad \quad \quad \quad \quad \quad \quad \leq3D_{\hat{k}} + 3D_{\hat{j}} + 2\left( \frac{3}{5}D_{\hat{k}} + \frac{3}{5}D_{\hat{j}} + \frac{2}{5}D_{\hat{t}}\right)=$$
$$= \frac{21}{5}D_{\hat{k}} + \frac{21}{5}D_{\hat{j}} + \frac{4}{5}D_{\hat{t}}, \quad$$
hence $$\frac{21}{5} D_{\hat{t}}\leq \frac{21}{5} D_{\hat{k}}+ \frac{21}{5} D_{\hat{j}}.$$
Obviously condition~(i) does not imply the strict triangle inequalities: 
in fact, if we take $D_{\hat{4}}= D_{\hat{3}} = D_{\hat{1}} + D_{\hat{2}}$,
condition~(i) is satisified.
\end{proof}

{\em Proof of Theorem~\ref{thm:morevertices}.}
$\Rightarrow $
Let ${\cal G}=(G,w)$
be a weighted graph  and let $[4] \subset V(G)$. 
Let $D_{\hat{t}}  ({\cal G})= D_{\hat{t}} $  for any distinct $ t \in [4]$. 
The  subgraph realizing $D_{\hat{t}}({\cal G})$ 
 is necessarily a tree with two or 
three  leaves. For any $t\in [4]$,  we choose a subtree realizing  
$D_{\hat{t}}({\cal G}) $.

%\begin{itemize}

We call $a_{1},a_{2},a_{3}$  the weights of the 
branches of the subtree realizing $D_{\hat{4}}({\cal G}) $ if it is a tree with three leaves,
$1,2, 3$. If it is a tree with two leaves, for instance if it is a tree
with leaves $1$ and $3$, we call $ a_1$ the weight of the path between
$1$ and $ 2$, $a_3$ the weight of the path between
$2$ and $ 3$  and we set $a_2=0$ and analogously in the other cases (see 
\figurename~\ref{fig:a1,a2,a3}).

\begin{figure}[htbp]
\centering
\input{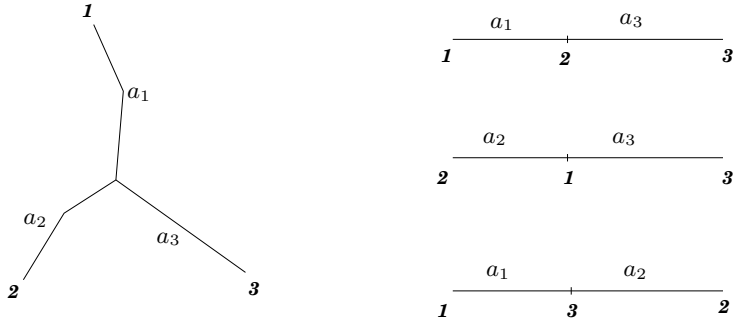}
\caption{$a_1, a_2, a_3$  \label{fig:a1,a2,a3} }
\end{figure}

In an analogous way we call  $b_{1},b_{2},b_{4}$  the weights of the branches of the subtree realizing $D_{\hat{3}} ({\cal G}) $, we call  $c_{1},c_{3},c_{4}$  the weights of the branches of the subtree realizing $D_{\hat{2}} ({\cal G})$,
 and, finally, we call $d_{2}, d_{3},d_{4}$  the weights of the branches of the subtree realizing $D_{\hat{1}}({\cal G}) $.  

%\item[-] We call $b_{1},b_{2},b_{4}$  the weights of the branches of the subtree realizing $D_{\hat{3}}$ if it is a tree with three leaves ($1,2, 4$). If it is a tree with two leaves, for instance if it is a tree with leaves $1$ and $4$, we call $ b_1$ the weight of the path between $1$ and $ 2$, $b_4$ the weight of the path between $2$ and $ 4$  and we set $b_2=0$ and analogously the other cases.

%\item[-] We call $c_{1},c_{3},c_{4}$  the weights of the branches of the subtree realizing $D_{\hat{2}}$ if it is a tree with three leaves ($1,3, 4$). If it is a tree with two leaves, for instance if it is a tree with leaves $1$ and $4$, we call $ c_1$ the weight of the path between $1$ and $ 3$, $c_4$ the weight of the path between $3$ and $ 4$  and we set $c_3=0$ and analogously the other cases.

%\item[-] We call $d_{2},d_{3},d_{4}$  the weights of the branches of the subtree realizing $D_{\hat{1}}$ if it is a tree with three leaves ($2,3, 4$). If it is a tree with two leaves, for instance if it is a tree with leaves $2$ and $4$, we call $ d_2$ the weight of the path between $2$ and $ 3$, $d_4$ the weight of the path between $3$ and $ 4$  and we set $d_3=0$ and analogously the other cases.

To prove (i), up to permuting $1,2,3,4$, we can suppose $t=1, k=4, j=3,i=2$,   
so we have to prove that $$5D_{\hat{1}} \leq 3D_{\hat{4}}+3D_{\hat{3}}+2D_{\hat{2}}.$$

Up to swapping the vertices $4$ and $3$, 
%Let us consider two cases: \begin{enumerate} \item let 
we can suppose
$b_{1}\leq a_{1}$; thus we have:
$$ 5D_{\hat{1}} \leq (D_{\hat{4}} +b_{1}+b_{4})+(D_{\hat{4}}+b_{1}+b_{4})+(D_{\hat{3}}-b_{1}+a_{2}+a_{3})+$$ $$+(D_{\hat{2}}+b_{1}+b_{2})+(D_{\hat{2}}+b_{1}+b_{2}) =\quad \quad\quad \quad\quad \quad\quad \, \,\,$$
$$=2D_{\hat{4}}+D_{\hat{3}}+2D_{\hat{2}}+(3b_{1}+2b_{2}+2b_{4}+a_{2}+a_{3})\leq$$
$$\quad \quad \leq 2D_{\hat{4}}+D_{\hat{3}}+2D_{\hat{2}}+(2b_{1}+2b_{2}+2b_{4}+a_{1}+a_{2}+a_{3})=$$
$$= 3D_{\hat{4}}+3D_{\hat{3}}+2D_{\hat{2}}; \quad \quad\quad \quad\quad \quad\quad \quad \quad \quad \quad\quad \quad \,\,\,$$

%\item let $a_{1}\geq b_{1}$, then we have:
%$$ 5D_{\hat{1}} \leq (D_{\hat{4}}-a_{1}+b_{2}+b_{4})+(D_{\hat{3}}+a_{1}+a_{3})+(D_{\hat{3}}+a_{1}+a_{3})+$$ $$+(D_{\hat{2}}+a_{1}+a_{2})+(D_{\hat{2}}+a_{1}+a_{2}) =\quad \quad\quad \quad\quad \quad\quad \, \,\,$$
%$$=D_{\hat{4}}+2D_{\hat{3}}+2D_{\hat{2}}+(3a_{1}+2a_{2}+2a_{3}+b_{2}+b_{4})\leq$$
%$$\quad \quad \leq D_{\hat{4}}+2D_{\hat{3}}+2D_{\hat{2}}+(2a_{1}+2a_{2}+2a_{3}+b_{1}+b_{2}+b_{4}) =$$
%$$= 3D_{\hat{4}}+3D_{\hat{3}}+2D_{\hat{2}}. \quad \quad\quad \quad\quad \quad\quad \quad \quad \quad \quad\quad \quad \,\,\, $$
%\end{enumerate}

Observe that all the inequalities hold also if some of the $a_i$ 
or some of the $b_i$ is zero. So we have proved (i).

To prove (ii),  up to permuting $1,2,3,4$,
 we can suppose   $i=4, k=1, j=2$. 
So  we have to prove that $$ D_{\hat{4}} < D_{\hat{1}} + D_{\hat{2}}.$$

If both the tree realizing $D_{\hat{1}} $ and the tree realizing
$D_{\hat{2}}$ have three leaves, we get:
$$  D_{\hat{4}} \leq d_2 + d_3 + c_3 + c_1 < D_{\hat{1}} + D_{\hat{2}}.$$

If  the tree realizing $D_{\hat{1}} $ has three leaves 
and the tree realizing
$D_{\hat{2}} $ has two leaves, we get:
$$  D_{\hat{4}} \leq d_2 + d_3 + c_3 + c_1 < D_{\hat{1}} + D_{\hat{2}}$$
if the leaves of the tree realizing $D_{\hat{2}}$ are $1,3$,
$$  D_{\hat{4}}  \leq d_2 + d_3 + c_3  < D_{\hat{1}} + D_{\hat{2}} $$
if the leaves of the tree realizing $D_{\hat{2}} $ are $4,3$,
$$  D_{\hat{4}}  \leq d_2 + d_3 + c_1 < D_{\hat{1}} + D_{\hat{2}}$$
if the leaves of the tree realizing $D_{\hat{2}}$ are $1,4$.

Analogously we can argue if  the tree realizing $D_{\hat{1}} $ has two leaves 
and the tree realizing $D_{\hat{2}}$ has three leaves.

Finally, if both the tree realizing $D_{\hat{1}}$ 
and the tree realizing $D_{\hat{2}}$ have two leaves, we get:
$$  D_{\hat{4}} \leq d_2 + d_3 + c_1 < D_{\hat{1}} + D_{\hat{2}}$$
if the leaves of the tree realizing $D_{\hat{2}}$ are $1,3$ and
the leaves of the tree realizing $D_{\hat{1}}$ are $2,3$
and analogously we can argue in the other cases.

$ \Leftarrow $  
We can suppose 
$D_{\hat{4}} \geq D_{\hat{3}} \geq D_{\hat{2}} \geq D_{\hat{1}}$.
Let $ {\cal G}=(G,w) $ be the weighted graph 
in \figurename~\ref{graphcase4},
  
\begin{figure}[htbp]
\centering
\input{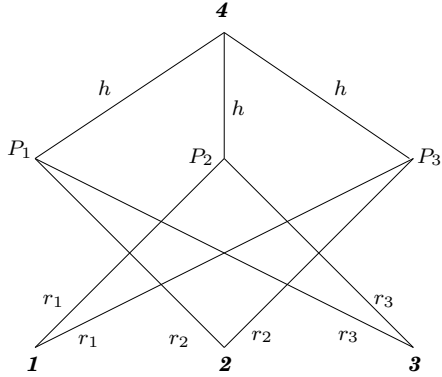}
\caption{A graph realizing 
$ (D_{\hat{1}}, D_{\hat{2}}, D_{\hat{3}}, D_{\hat{4}})$. 
\label{graphcase4}}
\end{figure}

where:

\medskip
 
$ h:= \frac{ D_{\hat{1}} +D_{\hat{2}}- D_{\hat{4}} }{2}$, 

\medskip

$r_1: = \frac{D_{\hat{4}} +D_{\hat{2}} + 2 D_{\hat{3}}  - 3 D_{\hat{1}} }{4} $,

\medskip

$r_2 := \frac{D_{\hat{4}} +D_{\hat{1}} + 2 D_{\hat{3}}  - 3 D_{\hat{2}} }{4} $,

\medskip

$r_3 := \frac{D_{\hat{4}} +D_{\hat{1}} +  D_{\hat{2}}  - 2 D_{\hat{3}} }{4} $,

\medskip

that is, $G$ is a graph with vertices $1,2,3,4, P_1, P_2 ,P_3$ and, for any 
$i \in [3]$,  
we join  the vertices $4$ and $P_i$ with an edge of weight $h$ and,
 for  any $i \in [3], j \in [3]-\{i\}$,  
we join the vertices $i$ and $P_j$  with an edge of weight $r_i$.

\medskip

From assumption (ii) and from the fact that $D_{\hat{4}} \geq D_{\hat{3}} \geq D_{\hat{2}} \geq D_{\hat{1}}$,
we get easily that  $h, r_1, r_2, r_3 $ are positive.
Besides observe that   $ 2h \geq r_3$ (by condition~(i)), hence: 

\medskip

$ D_{\hat{1}} ({\cal G}) = h + r_2 +  r_3 = D_{\hat{1}} $,

\medskip

$ D_{\hat{2}} ({\cal G}) = h + r_1 +  r_3 = D_{\hat{2}} $,

\medskip

$ D_{\hat{3}} ({\cal G}) = h + r_1 +  r_2 = D_{\hat{3}} $,

\medskip

$ D_{\hat{4}} ({\cal G}) = r_1 + r_2 + 2 r_3 = D_{\hat{4}} $.

\hfill \framebox(7,7)

\bigskip

{\bf Address of both authors:}
Dipartimento di Matematica e Informatica ``U. Dini'', 
viale Morgagni 67/A,
50134  Firenze, Italia

{\bf
E-mail addresses:}
baldisser@math.unifi.it, rubei@math.unifi.it


{\small \begin{thebibliography}{Dilloo Dilloo 83}


\bibitem[B-S]{B-S} H-J Bandelt, M.A. Steel {\em Symmetric matrices 
representable by weighted trees over a cancellative abelian monoid}.
SIAM J. Disc. Math. 8 (1995), no. 4, 517--525


\bibitem[B]{B} P. Buneman {\em A note on the metric properties
 of trees}. Journal of  Combinatorial Theory Ser. B  17 (1974), 48-50


\bibitem[H-H-M-S]{H-H-M-S} 
S.Herrmann, K.Huber, V.Moulton, A.Spillner, 
{\em Recognizing treelike k-dissimilarities}.
Journal of Classification 29  (2012),   no. 3, 321-340

\bibitem[H-P]{P-H} 
S.L. Hakimi, A.N. Patrinos {\em 
 The distance matrix of a graph and its tree realization}. 
Quart. Appl. Math. 30 (1972/73), 255-269


\bibitem[H-Y]{H-Y}
S.L. Hakimi, S.S. Yau,
{\em Distance matrix of a graph and its realizability}.
Quart. Appl. Math. 22 (1965), 305-317


\bibitem[P-S]{P-S}  L. Pachter, D. Speyer {\em Reconstructing
 trees from subtree weights}. Appl. Math. Lett. 17  (2004), no. 6, 615--621


\bibitem[Ru1]{Ru1}
E. Rubei {\em Sets of double and
 triple weights of trees}. Annals of Combinatorics 15  (2011), no. 4, 723-734  

\bibitem[Ru2]{Ru} E. Rubei {\em On dissimilarity vectors of general 
weighted trees}. Discrete Mathematics 312  (2012), no. 19,  2872-2880 


\bibitem[SimP]{SimP} J.M.S. Simoes Pereira {\em A Note on the Tree 
Realizability of a distance matrix}. Journal of Combinatorial Theory 6 (1969),
 303-310


\bibitem[S-S]{SS2}
D. Speyer, B. Sturmfels {\em Tropical mathematics}. 
Math. Mag. 82  (2009), no. 3, 163-173



\bibitem[St]{St} 
E. D. Stotskii {\em
 Embedding of Finite Metrics in graphs} Siberian Math. J. 5 (1964), 1203-1206


\end{thebibliography}}
\end{document}